\documentclass[12pt]{article}
\usepackage[colorlinks]{hyperref}
\usepackage{color}
\usepackage{graphicx}
\usepackage{graphics}
\usepackage{makeidx}
\usepackage{showidx}
\usepackage{latexsym}
\usepackage{amssymb}
\usepackage{verbatim}
\usepackage{amsmath}
\usepackage{amsthm}
\usepackage{amsfonts}
\usepackage{amssymb,amsmath}
\usepackage{latexsym,amsthm,amscd}
\usepackage[all]{xy}
\newtheorem{them}{Theorem}[section]
\newtheorem{prop}{Proposition}[section]
\newtheorem{defn}{Definition}[section]
\newcounter{alphthm}
\newtheorem{Lemma}[alphthm]{Lemma}
\newtheorem{propriete}[alphthm]{Theorem}

\numberwithin{equation}{section}
\newtheorem{cor}{Corollary}[section]

\newtheorem{rem}{Remark}[section]

\newtheorem{thm}{Theorem}
\newtheorem{lem}{Lemma}[section]
\newcommand{\be}{\begin{equation}}
\newcommand{\ee}{\end{equation}}
\newcommand{\ben}{\begin{enumerate}}
\newcommand{\een}{\end{enumerate}}
\newcommand{\beq}{\begin{eqnarray}}
\newcommand{\eeq}{\end{eqnarray}}
\newcommand{\beqn}{\begin{eqnarray*}}
\newcommand{\eeqn}{\end{eqnarray*}}

\newcommand{\bpf}{\begin{proof}}
\newcommand{\epf}{\end{proof}}
\newcommand{\bl}{\begin{lem}}
\newcommand{\el}{\end{lem}}
\newcommand{\bp}{\begin{prop}}
\newcommand{\ep}{\end{prop}}
\newcommand{\bd}{\begin{defn}}
\newcommand{\ed}{\end{defn}}
\newcommand{\bt}{\begin{thm}}
\newcommand{\et}{\end{thm}}

\def\nn{\nonumber}

\newcommand\bpr{\begin{prop}}
\newcommand\epr{\end{prop}}

\title{Hamilton's Ricci Flow on Finsler Spaces} 
\author{B. Bidabad\thanks{The corresponding author, bidabad@aut.ac.ir}\,\,and\,M. K. Sedaghat}
\date{}

\begin{document}
\maketitle
\noindent

\begin{abstract}
Recently, we have studied evolution of a family of Finsler metrics along Finsler Ricci flow and proved its convergence in short time. Here, existence of solutions to the so called Hamilton Ricci flow on Finsler spaces is studied and a short time solution is found. To this end the Finslerian Ricci-DeTurck flow on Finsler spaces is defined and existence of its solution in short time is proved. Next, this solution is pulled back to determine a short time solution to the Hamilton Ricci flow on underlying  Finsler space.
 \end{abstract}
\vspace{.5cm}
\section{Introduction}
The major aim for many geometric flows is to produce canonical geometric structures by deforming rather general initial data to these structures. The Ricci flow theory and its various applications became one of the most intensively developing branch of modern mathematics \cite{Cao,Chow2,MG}. The most fabulous achievement of this theory was the proof of Thurston's geometrization conjecture by G. Perelman \cite{Pr1,Pr2}. The subject of Hamilton's Ricci flow introduced in 1982 lies in the field of geometric flows which has many applications in physics and real world problems. Hamilton introduced the Ricci flow by the differential equation
\be \label{0}
\frac{\partial}{\partial t}g_{ij}(x,t)=-2Ric_{ij},\hspace{0.5cm}g_{(t=0)}=g_{0},\hspace{0.5cm}\forall x\in M,
\ee
in order to study compact three-manifolds with positive Ricci curvature. The Ricci flow, which evolves a Riemannian metric by its Ricci curvature, is a natural analogue of the heat equation for metrics. Ricci flow conformally deforms the Riemannian metric to its induced curvature, such that the curvature tensors evolve by a system of diffusion equations which leads to distribute the curvature uniformly over the manifold. Hence, one expects that the initial metric should be improved and evolved into a canonical metric, in order to obtain some information on the topology of the underlying manifold. In his celebrated paper, Hamilton proved that there is a unique solution to the Ricci flow for an arbitrary smooth Riemannian metric on a closed manifold over a sufficiently short time \cite{Ham1}. He also showed that on a compact three-manifold with an initial metric having positive Ricci curvature, the Ricci flow converges, after re-scaling to keep constant volume, to a metric of positive constant sectional curvature, proving the manifold is diffeomorphic to the three-sphere $\mathbb{S}^{3}$ or a quotient of the three-sphere $\mathbb{S}^{3}$ by a linear group of isometries.\\
In Finsler geometry the problem of constructing the Finslerian Ricci flow contains a number of new conceptual and fundamental issues on compatibility of geometrical and physical objects and their optimal configurations. In order to define the concept of Ricci tensor in Finsler geometry, Akbar-Zadeh in his celebrated work has used Einstein-Hilbert's functional in general relativity and introduced definition of Einstein-Finsler spaces as critical points of this functional, similar to the Hamilton's work, in Riemannian geometry, see \cite{Ham1}. This definition of Ricci tensor turned to be one of the most productive definition in Finsler geometry. In \cite{Bao}, D. Bao based on the Akbar-Zadeh's Ricci tensor and in analogy with the Ricci flow in Riemannian case, has considered the following equation as Ricci flow in Finsler geometry
\begin{eqnarray*}
\frac{\partial}{\partial t}\log F=-\mathcal{R}ic,\qquad F_{(t=0)}=F_{0},
\end{eqnarray*}
where, $F_{0}$ is the initial Finsler structure. This equation is equivalent to
\begin{eqnarray*}
\frac{\partial}{\partial t}g_{ij}(x,y,t)=-2Ric_{ij},\hspace{0.5cm}g_{(t=0)}=g_{0},\hspace{0.5cm}(x,y)\in TM,
\end{eqnarray*}
which is analogous to the Riemannian Ricci flow (\ref{0}) and addresses the evolution of the Finsler structure $F$. This definition of Ricci flow seems to make sense, as an un-normalized Ricci flow for Finsler spaces on both the manifolds of nonzero tangent vectors $TM_{0}$ and the sphere bundle $SM$, which is also used in \cite{SL} to study Harnack estimates for heat equation.\\

Recently, we have studied Finsler Ricci solitons as a self similar solutions to the Finsler Ricci flow and it was shown if there is a Ricci soliton on a compact Finsler manifold then there exists a solution to the Finsler Ricci flow equation and vice-versa, see \cite{BY}.
 Next, as a first step to answer  Chern's question stating that weather there exists a Finsler-Einstein metric on every smooth manifold, we have considered evolution of a family of Finsler metrics, first under a general flow next under Finsler Ricci flow and prove that a family of Finsler metrics $g(t)$ which are solutions to the Finsler Ricci flow converge to a smooth limit Finsler metric as $t$ approaches the finite time $T$, see \cite{YB2}.
Moreover, a Bonnet-Myers type theorem was studied and it is proved that on a Finsler space, a forward complete shrinking Ricci soliton is compact if and only if the corresponding vector field is bounded, using which we have shown a compact shrinking Finsler Ricci soliton has finite fundamental group and hence the first de Rham cohomology group vanishes, see \cite{YB1}. 

In the present work, a new fundamental step is taken in the study of any system of evolutionary partial differential equations which shows short-time existence of Hamilton Ricci flow on Finsler spaces. More intuitively, as the Ricci flow quasi-linear differential equation is not parabolic, one cannot apply the standard theory to conclude existence and uniqueness of a short time solution. Following the procedure described by D. DeTurck in Riemannian space \cite{DeT}, the Finslerian Ricci flow is modified and a PDE which has a short-time solution is obtained. Next, a solution to the original Ricci flow equation is found by pulling back solution of the modified flow using appropriate diffeomorphisms.
\section{Preliminaries and notations}

\subsection{Finsler space and Cartan connection}
Let $M$ be a real n-dimensional manifold of class $C^{\infty}$. We denote by $TM$ the tangent bundle of tangent vectors,  by  $\pi :TM_{0}\longrightarrow M$ the fiber bundle of non-zero tangent vectors and  by $\pi^*TM\longrightarrow TM_0$ the pulled-back tangent bundle.
Let $(x,U)$ be a local chart on  $ M$ and $(x^i,y^i)$ be the induced local coordinates on $\pi^{-1}(U)$.  A \emph{Finsler structure} on M is a function $F: TM \longrightarrow [0,\infty )$, with the following properties:(i) $F$ is differentiable $C^{\infty}$ on $TM_{0}$; (ii) $F$ is positively homogeneous of degree one in $y$, that is,  $F(x,\lambda y)=\lambda F(x,y)$, for all $\lambda >0$;  (iii) The Finsler metric tensor $g$ defined by the Hessian matrix of $F^{2}$, $(g_{ij})=(\frac{1}{2}[\frac{\partial^{2}}{\partial y^{i}\partial y^{j}}F^{2}])$, is positive definite on $TM_{0}$. A \emph{Finsler manifold} is a pair $(M,F)$ consisting of a differentiable manifold $M$ and a Finsler structure $F$ on $M$. Here and everywhere in this paper all the Latin indices $i,j,...,a,b,...,p,q,...$ run over the range $1,...,n$.
Any point of $TM_0$ is denoted
by $z=(x,y)$, where $x=\pi z\in M$ and $y\in T_{\pi z}M$. We denote by $TTM_0$, the tangent bundle of $TM_0$ and by $\varrho$, the canonical
linear mapping $\varrho:TTM_0\longrightarrow \pi^*TM,$ where, $ \varrho=\pi_*$. For all $z\in TM_0$, let $V_zTM$ be the set of vertical vectors at $z$, that is, the set of vectors which are tangent to the fiber through $z$. Equivalently, $V_zTM=\ker \pi_*$
where $\pi_*:TTM_0\longrightarrow TM$ is the linear tangent mapping.\\
Let $\nabla$ be a linear connection on the vector bundle $\pi^*TM\longrightarrow TM_0$. Consider the linear mapping
$\mu:TTM_0\longrightarrow \pi^*TM,$ by $\mu(\hat{X})=\nabla_{\hat{X}}{\bf y}$ where, $\hat{X}\in TTM_0$ and ${\bf y}$ is the canonical section of $\pi^*TM$.
The connection $\nabla$ is said to be {\it regular}, if $\mu$ defines an isomorphism between $VTM_0$ and
$\pi^*TM$. In this case, there is the horizontal distribution $HTM$ such that we have the Whitney sum $TTM_0=HTM\oplus VTM.$ This decomposition permits to write a vector field $\hat{X}\in \chi(TM_0)$ into the horizontal and vertical form $\hat{X}=H\hat{X}+V\hat{X}$ uniquely. In the sequel, we denote all vector fields on $TM_0$ by $\hat{X}, \hat{Y}$, etcetera and the corresponding sections of $\pi^*TM$ by $X=\varrho(\hat X)$, $Y=\varrho(\hat Y)$, respectively, unless otherwise specified.
The structural equations of the regular connection $\nabla$ are given by:
\begin{eqnarray*}
&&\tau(\hat{X},\hat{Y})=\nabla_{\hat{X}}Y-\nabla_{\hat{Y}}X-\varrho[\hat{X},\hat{Y}],\\
&&\Omega(\hat{X},\hat{Y})Z=\nabla_{\hat{X}}\nabla_{\hat{Y}}Z-\nabla_{\hat{Y}}\nabla_{\hat{X}}Z
-\nabla_{[\hat{X},\hat{Y}]}Z,
\end{eqnarray*}
where, $X=\varrho(\hat{X})$, $Y=\varrho(\hat{Y})$, $Z=\varrho(\hat{Z})$ and $\hat{X}$, $\hat{Y}$ and $\hat{Y}$ are vector fields on $TM_0$. The torsion tensor $\tau$ and curvature tensor $\Omega$ of $\nabla$ determine two torsion tensors denoted here by $S$ and $T$ and three
curvature tensors denoted by $R$, $P$ and $Q$ defined by:
\begin{eqnarray*}
&S(X,Y)=\tau(H\hat{X},H\hat{Y}),&\ \ \ T(\dot{X},Y)=\tau(V\hat{X},H\hat{Y}),\\
&R(X,Y)=\Omega(H\hat{X},H\hat{Y}),&\ \ \ P(X,\dot{Y})=\Omega(H\hat{X},V\hat{Y}),\\
&Q(\dot{X},\dot{Y})=\Omega(V\hat{X},V\hat{Y}),&\nonumber
\end{eqnarray*}
where, $X=\varrho(\hat{X})$,\ $Y=\varrho(\hat{Y})$,\
$\dot{X}=\mu(\hat{X})$ and $\dot{Y}=\mu(\hat{Y})$. The tensors $R$, $P$ and $Q$ are called $hh-$, $hv-$ and $vv-$curvature tensors, respectively.
There is a unique regular connection associated with $F$ called {\it Cartan connection}  such that:
\beq\label{CartanConProperties}
&& \nabla_{\hat{Z}}g=0,\nonumber\\
&& S(X,Y)=0,\nonumber\\
&& g(\tau(V\hat{X},\hat{Y}),Z)=g(\tau(V\hat{X},\hat{Z}),Y),
  \eeq
where, $X=\varrho(\hat{X})$, $Y=\varrho(\hat{Y})$ and $Z=\varrho(\hat{Z})$, for all $\hat{X}$, $ \hat{Y}$,
$\hat{Z}\in TTM_0$, see \cite{AZ}. Given an induced natural coordinates on $\pi^{-1}(U)$,
we denote by $G^{i}$ the components of spray vector field on $TM$, where $G^{i}=\frac{1}{4}g^{ih}(\frac{\partial^{2}F^{2}}{\partial y^{h}\partial x^{j}}y^{j}-\frac{\partial F^{2}}{\partial x^{h}})$, and the \emph{formal} \emph{Christoffel} \emph{symbols} by $\gamma^{i}_{jk}=\frac{1}{2}g^{ih}(\partial_{j}g_{hk}+\partial_{k}g_{jh}-\partial_{h}g_{jk})$. The horizontal and vertical subspaces have the corresponding bases $\{\frac{\delta}{\delta {x^i}},\frac{\partial}{\partial y^{i}}\}$, which are related to the typical bases of $TM$ $\{\frac{\partial}{\partial x^{i}},\frac{\partial}{\partial y^{i}}\}$, by $\frac{\delta}{\delta {x^i}}:=\frac{\partial}{\partial x^{i}}-G_{i}^{j}\frac{\partial}{\partial y^{j}}$. The dual bases of the former basis are denoted by $\{dx^{i},\delta y^{i}\}$, where $\delta y^{i}:=dy^{i}+G_{j}^{i}dx^{j}$. The 1-form of Cartan connection in these bases is given by $\omega^{i}_{j}=\Gamma^{i}_{jk}dx^{k}+C^{i}_{jk}\delta y^{k}$, where $\Gamma^{i}_{jk}=\frac{1}{2}g^{ih}(\delta_{j}g_{hk}+\delta_{k}g_{jh}-\delta_{h}g_{jk})$, $C^{i}_{jk}=\frac{1}{2}g^{ih}\dot{\partial}_{h}g_{jk}$, $\delta_{k}=\frac{\delta}{\delta x^{k}}$ and $\dot{\partial}_{k}=\frac{\partial}{\partial y^{k}}$. By homogeneity, we have $y^{k}\Gamma^{i}_{jk}=N^{i}_{j}$, where $N_{j}^{i}=\frac{1}{2}\frac{\partial G^{i}}{\partial  y^{j}}$ and $y^{j}N^{i}_{j}=2G^{i}$, see \cite{BCS}. The horizontal and vertical metric compatibility of  Cartan connection  in local coordinates is given by $\nabla_{l}g_{jk}=0$ and $\dot{\nabla}_{l}g_{jk}=0$ respectively. In local coordinates, coefficients of the Cartan connection $\nabla$ are given by
$$\nabla_{k}\dot{\partial}_{j}=\Gamma^{i}_{jk}\dot{\partial}_{i},\hspace{0.4cm}
\dot{\nabla}_{k}\dot{\partial} _{j}=C^{i}_{jk}\dot{\partial}_{i},\hspace{0.4cm}
\nabla_{k}\delta_{j}=\Gamma^{i}_{jk}\delta_{i},\hspace{0.4cm} \dot{\nabla}_{k}\delta_{j}=C^{i}_{jk}\delta_{k}.$$
The components of Cartan horizontal and vertical covariant derivatives of a Finslerian $(1,2)$ tensor field $S$ on $\pi^{*}TM$ with the components $(S^{i}_{jk}(x,y))$ on $TM$ are given by
\begin{equation} \label{1}
\nabla_{l}S^{i}_{jk}:= \delta_{l}S^{i}_{jk}-S^{i}_{s k}\Gamma^{s}_{jl}-S^{i}_{js}\Gamma^{s}_{kl}+S^{s}_{jk}\Gamma^{i}_{s l},
\end{equation}
\begin{equation}\label{2}
\dot{\nabla}_{l}S^{i}_{jk}:=\dot{\partial}_{l}S^{i}_{jk}-S^{i}_{s k}C^{s}_{jl}-S^{i}_{js}C^{s}_{kl}+S^{s}_{jk}C^{i}_{s l},
\end{equation}
respectively, where, $\nabla_{l}:=\nabla_{\frac{\delta}{\delta x^l}}$ and $\dot{\nabla}_{l}:=\nabla_{\frac{\partial}{\partial y^l}}$.
We denote the components of Cartan hh-curvature tensor by
\begin{equation} \label{77}
R^{i}_{jkm}=\delta_{k}\Gamma^{i}_{jm}-\delta_{m}\Gamma^{i}_{jk}+
\Gamma^{i}_{s k}\Gamma^{s}_{jm}-\Gamma^{i}_{s m}\Gamma^{s}_{jk}
+R^{s}_{km}C^{i}_{s j},
\end{equation}
where, $R^{s}_{km}$ is equal to $y^{p} R^{s}_{ pkm}$.\\
Here, we consider also the reduced curvature tensor $R^{i}_{k}$ which is expressed entirely in terms of $x$ and $y$ derivatives of spray coefficients $G^{i}$ as follows, see \cite{BCS}
\begin{equation} \label{18}
R^{i}_{k}:=\frac{1}{F^2}(2\frac{\partial G^{i}}{\partial x^{k}}-\frac{\partial^{2}G^{i}}{\partial x^{j}\partial y^{k}}y^{j}+2G^{j}\frac{\partial^{2}G^{i}}{\partial y^{j}\partial y^{k}}-\frac{G^{i}}{\partial y^{j}}\frac{G^{j}}{\partial y^{k}}).
\end{equation}
Let  $c:I\longrightarrow M$ be an oriented $C^{\infty}$ parametric curve on $(M,F)$ with the parametric equation $x^{i}(t)$. Let $(x(t),\dot{x}(t))$ be the line element along the curve $c$ and $X$ a $C^{\infty}$ vector field $X=X^{i}(t)\frac{\partial}{\partial x^{i}}\mid_{c(t)}$ along $c(t)$. We denote the Cartan covariant derivative of $X$ in direction of $\dot{c}=\frac{dx^{j}}{dt}\frac{\partial}{\partial x^{j}}$ by $\nabla_{\dot{c}}X=\frac{\delta X^{i}}{dt}\frac{\partial}{\partial x^{i}}\mid_{c(t)}$, see \cite{BS}, where
\begin{equation} \label{8}
\frac{\delta X^{i}}{dt}=\frac{dX^{i}}{dt}+(\Gamma^{i}_{kh}+C^{i}_{ks}N^{s}_{h})X^{k}\frac{dx^{h}}{dt}.
\end{equation}
By means of metric-compatibility we have,
\begin{eqnarray*}
\frac{d}{dt}g(X,Y)=g(\nabla_{\dot{c}}X,Y)+g(X,\nabla_{\dot{c}}Y).
\end{eqnarray*}
\subsection{On the pull back bundle $p^{*}TM$ over $SM$}
Consider the sphere bundle $SM:=TM/\sim$, where $y\sim y'$ if and only if $y=\lambda y'$ for some $\lambda>0$. Given any $(x,y)\in TM$, we shall denote its equivalence class which is a point in $SM$ by $(x,[y])\in SM$.
The natural projection $p:SM\longrightarrow M$ pulls back the tangent bundle $TM$ to a n-dimensional vector bundle $p^{*}TM$ over the $2n-1$ dimensional base $SM$. Namely, over each point $(x,[y])$ we erect a single copy of $T_{x}M$ and endow it with the inner product $g_{ij}(x,y)dx^{i}\otimes dx^{j}$. The resulting vector bundle $p^{*}TM$ has the fiber dimension $n$ as in $\pi^*TM$, but now it sits over the $2n-1$ dimensional sphere bundle $SM$ rather than $TM$. Local coordinates $x^{1},...,x^{n}$ on $M$ induces global coordinates $y^1,...,y^n$ on each fiber $T_{x}M$, through the expansion $y=y^{i}\frac{\partial}{\partial x^i}$. Thus $(x^i;y^i)$ is a coordinate system on $SM$ with the $y^i$ regarded as homogeneous coordinates in the projective space sense.
Given local coordinates $(x^i)$ on $M$,
we shall economize on notation and regard the corresponding collections $\{\frac{\partial}{\partial x^i}\}$, $\{dx^i\}$ as local bases for $p^{*}TM$ and its dual $p^{*}T^{*}M$, respectively. There is an inner product $g$ on $p^{*}TM$ by specification
\begin{eqnarray*}
g:=g_{ij}(x,y)dx^i\otimes dx^{j},
\end{eqnarray*}
on the fiber over the point $(x,[y])\in SM$, where $y\in T_{x}M$. One can verify that $g$ is well defined and
\begin{eqnarray*}
g_{ij}(x,y)y^{i}y^{j}=F^{2}(x,y).
\end{eqnarray*}
Let $\{e_{a}=u_{a}^{i}\frac{\partial}{\partial x^i}\}$ be a local orthonormal frame for $p^{*}TM$ and $\{\omega^a=v^{a}_{i}dx^i\}$ be its co-frame; thus $\omega^{a}(e_{b})=\delta^{a}_{b}$. It is clear that $e_{n}:=l$, where $l$ is the distinguished global section whose value at any $(x,[y])\in SM$ and $\omega^{n}=\frac{\partial F}{\partial y^i}dx^{i}$.
Also we have $\frac{\partial}{\partial x^{i}}=v^{a}_{i}e_{a}$ and $dx^{i}=u^{i}_{a}\omega^{a}$. A basic relation between $(u^i_a)$ and $(v^a_i)$ is given by $v^a_iu_b^i=\delta^a_b$ and $u^i_av^a_j=\delta^i_j$.
For convenience, we shall also regard the $e_{a}$'s and $\omega^a$'s as local vector fields and 1-forms, respectively on $SM$. All $p^{*}TM$ related indices are raised and lowered with the metric $g$, see \cite{BAOS}.
 Let
\begin{eqnarray*}
&&\hat{e}_{a}=u^{i}_{a}\frac{\delta}{\delta x^{i}},\quad
\hat{e}_{n+a}=u^{i}_{a}F\frac{\partial}{\partial y^i}.\\
&&\omega^{a}=v^{a}_{i}dx^{i},\quad
\omega^{n+a}=v^{a}_{i}\frac{\delta y^i}{F}.
\end{eqnarray*}
It can be shown that $\{\hat{e}_{a},\hat{e}_{n+\alpha}\}$ is a local basis for the tangent bundle $TSM$ and $\{\omega^{a},\omega^{n+\alpha}\}$ is a local basis for the cotangent bundle $T^{*}SM$, where the Latin indices run over the range $1,...,n$ and the Greek indices run from $1$ to $n-1$. Tangent vectors on $SM$ which are annihilated by all $\{\omega^{n+\alpha}\}$'s form the horizontal sub-bundle $HSM$ of $TSM$. The fibers of $HSM$ are $n$-dimensional. On the other hand, let $VSM:=\cup_{x\in M}T(S_{x}M)$ be the vertical sub-bundle of $TSM$; its fibers are $n-1$ dimensional. The decomposition $TSM=HSM\oplus VSM$ holds because $HSM$ and $VSM$ are direct summands.
The inner product on $p^{*}TM$ induces a Riemannian structure $\hat{g}$ on $SM$
\begin{eqnarray*}
\hat{g}:=
\delta_{ab}\omega^{a}\otimes \omega^{b}+\delta_{\alpha\beta}\omega^{n+\alpha}\otimes\omega^
{n+\beta}=
g_{ij}dx^{i}\otimes dx^{j}+\delta_{\alpha\beta}\omega^{n+\alpha}\otimes\omega^
{n+\beta}.
\end{eqnarray*}
In particular, $HSM$ and $VSM$ are orthogonal with respect to $\hat{g}$.
\section{Quasilinear strictly parabolic equations on $SM$}
A system of partial differential equations is called \emph{quasilinear} if the derivatives of principal order term occur only linearly (with coefficients which may depend on derivatives of lower order),
see \cite{Rog}. Let $u:M\longrightarrow \mathbb{R}$ be a smooth function on $M$.
A \emph{quasilinear strictly parabolic} equation is a PDE of the form
\begin{eqnarray*}
\frac{\partial u}{\partial t}=a^{ij}(x,t)\frac{\partial^2 u}{\partial x^i\partial x^j}+h(x,t,u,\frac{\partial u}{\partial x^i}),
\end{eqnarray*}
where, $a^{ij}$ and $h$ are smooth functions on $M$ and for some constant $\lambda>0$ we have the parabolicity assumption
\begin{eqnarray*}
a^{ij}(x,t)\xi_{i}\xi_{j}\geq \lambda\parallel \xi\parallel^{2},\quad 0\neq\xi\in\chi(M),
\end{eqnarray*}
that is, all eigenvalues of $A=(a^{ij})$ have positive sign or equivalently $(a^{ij})$ is positive definite.
 Let $\rho:SM\longrightarrow \mathbb{R}$ be a smooth function on the sphere bundle $SM$. We consider the quasilinear strictly parabolic equation on $SM$
\begin{equation}\label{parabolic}
\frac{\partial \rho}{\partial t}=G^{_{AB}}(x,y,t)\hat{e}_{_A}\hat{e}_{_B}\rho+h(x,y,t,\rho,\hat{e}_{_A}\rho),
\end{equation}
where $\hat{e}_{_A}$ is a local basis for the tangent bundle $TSM$ and stand here as partial derivative on $SM$. The capital indices $\footnotesize{A,B}$ run over the range $\{1,2,...,2n-1\}$ and $G^{_{AB}}$ is positive definite. Here, $G^{_{AB}}$ and $h$ are smooth functions on $SM$. More preciesly, a quasilinear strictly parabolic equation on $SM$ can be defined in the following form
\begin{equation}\label{po}
\frac{\partial \rho}{\partial t}=P^{ab}\hat{e}_{a}\hat{e}_{b}\rho+Q^{\alpha\beta}\hat{e}_{n+\alpha}\hat{e}_{n+\beta}\rho+C^{a\alpha}\hat{e}_{a}\hat{e}_{n+\beta}+\textrm{lower order terms},
\end{equation}
where the Latin indices $a,b,...$ and Greek indices $\alpha,\beta,...$ run over the range $1,...,n$ and $1,...,(n-1)$, respectively and the matrix
\begin{displaymath}
G^{_{AB}}=\left(\begin{array}{c|c}
P^{ab} & \frac{1}{2}C^{a\alpha} \\
\hline
\frac{1}{2}C^{\alpha a} & Q^{\alpha\beta}
\end{array}\right)_{(2n-1)\times(2n-1)},
\end{displaymath}
is positive definite.
\begin{lem}\label{mm}
Let $\rho:TM\longrightarrow \mathbb{R}$ be a zero-homogeneouse smooth function on the tangent bundle $TM$. The quasilinear differential equation
\begin{equation}\label{lili}
\frac{\partial \rho}{\partial t}=g^{ij}\frac{\delta^2\rho}{\delta x^i\delta x^j}+F^2g^{ij}\frac{\partial^2 \rho}{\partial y^i\partial y^j}+\textrm{lower order terms},
\end{equation}
is a quasilinear strictly parabolic equation on $SM$.
\end{lem}
\section{The harmonic map Laplacian on the Tangent Bundle}
Let $\varphi:(M,g)\longrightarrow (N,h)$ be a diffeomorphism between two n-dimensional Finsler manifolds $(M,F)$ and $(N,\bar{F})$ with the corresponding metric tensors $g$ and $h$ called domain and codomain metric, respectively. Denote image of the linear tangent application $d\varphi(T_{x}M)$ by $T_{\bar{x}}N$, where $\varphi(x)=\bar{x}$. Let $c$ be a geodesic on $(M,F)$ and $\bar{c}(t):=\varphi \circ c(t)$ its image by $\varphi$.
Consider a local coordinate system $\{x^{p}\}$ on $M$, a local coordinate system $\{\bar{x}^{i}\}$ on $N$ and $\bar{y}=\dot{\bar{c}}(0)=d\varphi(y)$, where $y=\dot{c}(0)$. Then (\ref{8}) leads to
\begin{eqnarray}\label{12}
\bar{\nabla}_{\dot{\bar{c}}}\dot{\bar{c}}(0)&=&\{\frac{d^{2}\bar{c}^{i}}{dt^{2}}(0)+(\bar{\Gamma}^{i}_{kh}+
\bar{C}^{i}_{ks}\bar{N}^{s}_{h})\frac{d\bar{c}^k}{dt}(0)\frac{d\bar{x}^h}{dt}\}\frac{\partial}{\partial \bar{x}^{i}}\mid_{\bar{x}}\\
&=&\{\frac{d^{2}\bar{c}^{i}}{dt^{2}}(0)+(\bar{\Gamma}^{i}_{kh}+
\bar{C}^{i}_{ks}\bar{N}^{s}_{h})\bar{y}^{k}\bar{y}^{h}\}\frac{\partial}{\partial \bar{x}^{i}}\mid_{\bar{x}}\nonumber\\
&=&\{\frac{\partial^{2}\varphi^{i}}{\partial x^{p} \partial x^{q}}(x)y^{p}y^{q}+\frac{\partial \varphi^{i}}{\partial x^{p}}(x)\frac{d^{2}c^{p}}{dt^{2}}+(\bar{\Gamma}^{i}_{kh}+
\bar{C}^{i}_{ks}\bar{N}^{s}_{h})\bar{y}^{k}\bar{y}^{h}\}\frac{\partial}{\partial \bar{x}^{i}}\mid_{\bar{x}}\nonumber\\
&=&\{\frac{\partial^{2}\varphi^{i}}{\partial x^{p} \partial x^{q}}(x)y^{p}y^{q}
-\frac{\partial \varphi^{i}}{\partial x^{p}}(x)G^{p}(y)+(\bar{\Gamma}^{i}_{kh}+
\bar{C}^{i}_{ks}\bar{N}^{s}_{h})\bar{y}^{k}\bar{y}^{h}\}\frac{\partial}{\partial \bar{x}^{i}}\mid_{\bar{x}}.\nonumber
\end{eqnarray}
On the other hand, $\bar{y}^{i}=\frac{\partial \varphi^{i}}{\partial x^{p}}y^{p}$, thus (\ref{12}) is given by
\begin{eqnarray*}
\bar{\nabla}_{\dot{\bar{c}}}\dot{\bar{c}}(0)
&=&\{\frac{\partial^{2}\varphi^{i}}{\partial x^{p} \partial x^{q}}(x)y^{p}y^{q}
-\frac{\partial \varphi^{i}}{\partial x^{k}}(x)\Gamma^{k}_{pq}y^{p}y^{q}+\bar{\Gamma}^{i}_{kh}\bar{y}^{k}\bar{y}^{h}
\}\frac{\partial}{\partial \bar{x}^{i}}\mid_{\bar{x}}\nonumber \\
&=&y^{p}y^{q}\{\frac{\partial^{2}\varphi^{i}}{\partial x^{p} \partial x^{q}}(x)
-\frac{\partial \varphi^{i}}{\partial x^{k}}(x)\Gamma^{k}_{pq}
+\bar{\Gamma}^{i}_{kh}\frac{\partial \varphi^{k}}{\partial x^{p}}\frac{\partial \varphi^{h}}{\partial x^{q}}\}\frac{\partial}{\partial \bar{x}^{i}}\mid_{\bar{x}}.\nonumber
\end{eqnarray*}
Next, let
\begin{eqnarray*}
A^{i}_{pq}=\{\frac{\partial^{2}\varphi^{i}}{\partial x^{p} \partial x^{q}}(x)
-\frac{\partial \varphi^{i}}{\partial x^{k}}(x)\Gamma^{k}_{pq}
+\bar{\Gamma}^{i}_{kh}\frac{\partial \varphi^{k}}{\partial x^{p}}\frac{\partial \varphi^{h}}{\partial x^{q}}\},\nonumber
\end{eqnarray*}
where $\Gamma^{k}_{pq}$ and $\bar{\Gamma}^{i}_{kh}$ are the coefficients of horizontal covariant derivative of Cartan connection on the coordinate systems $\{x^{i}\}$ and $\{\bar{x}^{i}\}$ on $(M,F)$ and $(N,\bar{F})$, respectively.
Contracting $A^i_{pq}$ with $g^{pq}$ leads to the following definition
\begin{equation} \label{15}
(\Delta_{g,h}\varphi)^{i}:=g^{pq}\{\frac{\partial^{2}\varphi^{i}}{\partial x^{p} \partial x^{q}}(x)
-\frac{\partial \varphi^{i}}{\partial x^{k}}(x)\Gamma^{k}_{pq}
+\bar{\Gamma}^{i}_{kh}\frac{\partial \varphi^{k}}{\partial x^{p}}\frac{\partial \varphi^{h}}{\partial x^{q}}\},
\end{equation}
where, $(\Delta_{g,h}\varphi)^{i}=A^{i}_{pq}g^{pq}$.
The \emph{harmonic} \emph{map} \emph{Laplacian} of $\varphi$ with respect to the domain metric $g$ and the codomain metric $h$ is defined by
$$\Delta_{g,h}\varphi=(\Delta_{g,h}\varphi)^{i}\frac{\partial}{\partial \bar{x}^{i}},$$
where, $(\Delta_{g,h}\varphi)^{i}$ is given by (\ref{15}).
In the following proposition it is shown that the harmonic map Laplacian is invariant under the action of the diffeomorphism group of $M$.
\begin{prop} \label{main2}
Let $\varphi:(M,g)\longrightarrow (N,h)$ be a diffeomorphism between the two Finsler spaces $(M,F)$ and $(N,\bar{F})$ with the corresponding metric tensors $g$ and $h$, respectively. If $\psi$ is a diffeomorphism from $M$ to itself, then
\begin{eqnarray*}
\Delta_{\tilde{\psi}^{*}(g),h}(\varphi \circ\psi)\mid_{x}=(\Delta_{g,h}\varphi)\mid_{\psi(x)}\in T_{\bar{x}}N,\nonumber
\end{eqnarray*}
for all $x\in M$ where $\bar{x}=(\varphi \circ\psi)(x)$ and $\tilde{\psi}$ is the canonical lift of $\psi$ on $TM$.
\end{prop}
\begin{cor} \label{main3}
Let $(M,\tilde{g})$ and $(N,h)$ be two n-dimensional Finsler spaces with corresponding Finsler structures $\tilde{F}$ and $\bar{F}$ respectively. Let $N=M$ and $\varphi$ be the identity map $\varphi=Id:(M,\tilde{g})\longrightarrow (M,h)$, then we have
\begin{equation} \label{17}
(\Delta_{\tilde{g},h}Id)^{k}=\tilde{g}^{ij}(-\tilde{\Gamma}^{k}_{ij}+\bar{\Gamma}^{k}_{ij}),
\end{equation}
where, $\tilde{\Gamma}^{k}_{ij}$ and $\bar{\Gamma}^{k}_{ij}$ are the coefficients of horizontal covariant derivative of Cartan connection with respect to the $\tilde{g}$ and $h$, respectively.
\end{cor}
Next, consider the vector field $\xi(x,y)$ as a section of $\pi^{*}TM$ with the components $\xi^{k}:=(\Delta_{\tilde{g},h}Id)^{k}$.
It can be easily verified that the components of $\xi=\Delta_{\tilde{g}(t),h}Id=\tilde{g}^{mn}(x,y)(-\tilde{\Gamma}^{k}_{mn}+\bar{\Gamma}^{k}_{mn})\frac{\partial}{\partial x^k}\in \Gamma(\pi^{*}TM)$ are homogeneous of degree zero on $y$, thus $\xi$ can be considered as a vector field on $SM$.
 Using the fact that the difference of two connections is a tensor, $\xi$ is a globally well-defined vector field.
\section{ Finslerian Ricci-DeTurck flow and existence of solution}
In general Finslerian setting, a notable definition of Ricci tensor is introduced by H. Akbar-Zadeh as follows:
\begin{equation} \label{19}
Ric_{jk}:=[\frac{1}{2}F^{2}\mathcal{R}ic]_{y^{j}y^{k}},
\end{equation}
where, $\mathcal{R}ic=R^{i}_{i}$ and $R^{i}_{k}$ are defined by (\ref{18}). One of the advantages of Ricci quantity defined here is its independence to the choice of Cartan, Berwald or Chern(Rund) connections. Based on the Ricci tensor, in analogy with the Ricci-flow in Riemannian case, D. Bao has considered, the following natural extension of Ricci flow in Finsler geometry, see \cite{Bao},
\begin{equation}\label{20002}
\frac{\partial}{\partial t}g_{jk}=-2Ric_{jk},\hspace{0.6cm}g_{(t=0)}=g_{0}.
\end{equation}
Contracting with $y^{j}y^{k}$, via Euler's theorem, leads to $\frac{\partial}{\partial t}F^{2}=-2F^{2}\mathcal{R}ic$. That is,
\begin{equation} \label{20}
\frac{\partial}{\partial t}\log F=-\mathcal{R}ic,\hspace{0.6cm}F_{(t=0)}=F_{0},
\end{equation}
where $F_{0}$ is the initial Finsler structure. It can be easily verified that (\ref{20002}) and (\ref{20}) are equivalent. In order to show that the above Ricci flow in Finsler geometry has a solution on a short time interval, we replace the Finslerian Ricci flow by an equivalent evolution equation which has a solution in short time. In analogy with the Ricci-DeTurck flow in Riemannian case, we consider the following natural extension of Ricci-DeTurck flow in Finsler geometry.
\begin{defn} \label{Main4}
Let $(M,F)$ be a compact Finsler manifold with the fixed back-ground metric tensor $h$. Assume that for all $t\in [0,T)$, $\tilde{F}(t)$ is a one-parameter family of Finsler structures on $TM$ and $\tilde{g}(t)$ is the tensor metric related to $\tilde{F}(t)$. We say that $\tilde{g}(t)$ is a solution to the Finslerian Ricci-DeTurck flow if
\begin{equation} \label{21}
\frac{\partial}{\partial t}\tilde{g}_{jk}(t)=-2Ric_{jk}(\tilde{g}(t))-\mathcal{L}_{\xi}\tilde{g}_{jk}(t),
\end{equation}
where, $\mathcal{L}_{\xi}$ is the Lie derivative with respect to $\xi=\tilde{g}^{mn}(-\tilde{\Gamma}^{k}_{mn}+\bar{\Gamma}^{k}_{mn})\frac{\partial}{\partial x^k}\in \Gamma(\pi^{*}TM)$ which is a vector field on $SM$ as mentioned earlier.
\end{defn}
Contracting (\ref{21}) by $y^{j}y^{k}$, via Euler's theorem, we have the \emph{ scalar Ricci-DeTurck flow} in Finsler spaces.
\begin{equation} \label{22}
\frac{\partial}{\partial t}\tilde{F}^{2}(t)=-2\tilde{F}^{2}(t)\mathcal{R}ic(\tilde{g}(t))-\mathcal{L}_{\xi}\tilde{F}^{2}(t).
\end{equation}
By means of commutativity of $y$ and $\mathcal{L}_{\xi}$ one can easily verify that (\ref{22}) is equivalent to (\ref{21}). Moreover, this definition of Ricci-DeTutck flow reduces to the Ricci-DeTutck flow in Riemannian geometry if the Finsler structure $\tilde{F}$ arises from a Riemannian metric $\tilde{g}$ or equivalently $\tilde{g}_{jk}$ is independent of $y$.
The following theorem shows that the above Ricci-DeTurck flow is well defined and have a solution on a short time interval.
\begin{them}\label{main8}
Let $M$ be a compact differentiable manifold with a fixed background Finsler structure $\bar{F}$ and related Finsler metric $h$. Given any initial Finsler structure $F_{0}$ with metric tensor $g_{0}$, there exists a real number $T>0$ and a smooth one-parameter family of Finsler structures $\tilde{F}(t)$, $t\in[0,T)$, with metric tensor $\tilde{g}(t)$, such that $\tilde{F}(t)$ is a solution to the Finslerian Ricci-DeTurck flow and $\tilde{F}(0)=F_{0}$.
\end{them}
\begin{proof}
The Finslerian  Ricci-DeTurck flow (\ref{22}) can be written in the following form
\begin{equation} \label{28}
y^{p}y^{q}\frac{\partial}{\partial t}\tilde{g}_{pq}(t)=-2\tilde{F}^{2}(t)\mathcal{R}ic(\tilde{g}(t))-y^{p}y^{q}\mathcal{L}_{\xi}\tilde{g}_{pq}(t).
\end{equation}
By means of Lie derivative formula 
along the vector field $\xi$, contracting with $y^{p}y^{q}$ and using Euler's theorem we have
\begin{equation} \label{30}
y^{p}y^{q}\mathcal{L}_{\xi}\tilde{g}_{pq}=y^{p}y^{q}(\nabla_{p}\xi_{q}+\nabla_{q}\xi_{p}).
\end{equation}
We rewrite the term $\nabla_{p}\xi_{q}+\nabla_{q}\xi_{p}$ as follows
\begin{eqnarray*}
\nabla_{p}\xi_{q}+\nabla_{q}\xi_{p}=(\nabla_{p}\tilde{g}_{ql}\xi^{l})+(\nabla_{q}\tilde{g}_{pl}\xi^{l})=
\tilde{g}_{ql}(\nabla_{p}\xi^{l})+\tilde{g}_{pl}(\nabla_{q}\xi^{l}).
\end{eqnarray*}
Thus (\ref{30}) is written
\begin{eqnarray} \label{31}
&&y^{p}y^{q}\mathcal{L}_{\xi}\tilde{g}_{pq}=y^{p}y^{q}\tilde{g}^{mn}(\delta_{p}\delta_{q}\tilde{g}_{mn}-
\delta_{q}\delta_{n}\tilde{g}_{pm}-\delta_{p}\delta_{m}\tilde{g}_{qn}) \\
&&+y^{p}y^{q}[\tilde{g}_{ql}\delta_{p}(\tilde{g}^{mn})(\Gamma(h)^{l}_{mn}-\Gamma(\tilde{g})^{l}_{mn})+\tilde{g}_{pl}\delta_{q}(\tilde{g}^{mn})(\Gamma(h)^{l}_{mn}-\Gamma(\tilde{g})^{l}_{mn})\nonumber\\
&&+\tilde{g}_{ql}\tilde{g}^{mn}\delta_{p}
(\Gamma(h)^{l}_{mn})+\tilde{g}_{pl}\tilde{g}^{mn}\delta_{q}(\Gamma(h)^{l}_{mn})+\tilde{g}_{ql}\Gamma(\tilde{g})^{l}_{pw}\xi^{w}+\tilde{g}_{pl}\Gamma(\tilde{g})^{l}_{qw}\xi^{w}].\nonumber
\end{eqnarray}
Also we have
\begin{equation} \label{32}
-2\tilde{F}^{2}\mathcal{R}ic(\tilde{g})=-2\tilde{F}^{2}R^{s}_{s}=-2\tilde{F}^{2}l^{p}R^{s}_{psq}l^{q},
\end{equation}
where $R^{s}_{psq}$ are the components of hh-curvature tensor of Cartan connection and $l^{p}=\frac{y^{p}}{\tilde{F}}$ are the components of Liouville vector field. 
Hence,
\begin{eqnarray} \label{34}
&-2\tilde{F}^{2}\mathcal{R}ic(\tilde{g})&=-y^{p}y^{q}\tilde{g}^{sh}(\delta_{s}\delta_{p}\tilde{g}_{hq}-\delta_{s}\delta_{h}\tilde{g}_{pq}-\delta_{q}\delta_{p}\tilde{g}_{hs}+\delta_{q}\delta_{h}\tilde{g}_{ps})\nonumber\\
&&-2y^{p}y^{q}(\Gamma^{s}_{hs}\Gamma^{h}_{pq}-\Gamma^{s}_{hq}\Gamma^{h}_{ps})-y^{p}y^{q}(\delta_{s}\tilde{g}^{sh})(\delta_{p}\tilde{g}_{hq}+\delta_{q}\tilde{g}_{ph}-\delta_{h}\tilde{g}_{pq})\nonumber\\
&&+y^{p}y^{q}(\delta_{q}\tilde{g}^{sh})(\delta_{p}\tilde{g}_{hs}+\delta_{s}\tilde{g}_{ph}-\delta_{h}\tilde{g}_{ps}).
\end{eqnarray}
Substituting (\ref{31}) and (\ref{34}) in (\ref{28}), we obtain an equation which will be referred in the sequel as differential equation of scalar Finslerian Ricci-DeTurck flow
\begin{align} \label{36}
y^{p}y^{q}\frac{\partial}{\partial t}\tilde{g}_{pq}(t)&=y^{p}y^{q}\tilde{g}^{sh}(\delta_{s}\delta_{h}\tilde{g}_{pq})\\
&-2y^{p}y^{q}(\Gamma^{s}_{hs}\Gamma^{h}_{pq}-\Gamma^{s}_{hq}
\Gamma^{h}_{ps})-y^{p}y^{q}(\delta_{s}\tilde{g}^{sh})(\delta_{p}\tilde{g}_{hq}+\delta_{q}\tilde{g}_{ph}-\delta_{h}\tilde{g}_{pq})\nonumber\\
&+y^{p}y^{q}(\delta_{q}\tilde{g}^{sh})(\delta_{p}\tilde{g}_{hs}+\delta_{s}\tilde{g}_{ph}-\delta_{h}\tilde{g}_{ps})\nonumber\\
&-y^{p}y^{q}[\tilde{g}_{ql}(\delta_{p}\tilde{g}^{sh})(\Gamma(h)^{l}_{sh}-\Gamma(\tilde{g})^{l}_{sh})-\tilde{g}_{pl}(\delta_{q}\tilde{g}^{sh})(\Gamma(h)^{l}_{sh}-\Gamma(\tilde{g})^{l}_{sh})\nonumber\\
&+\tilde{g}_{ql}\tilde{g}^{sh}\delta_{p}(\Gamma(h)^{l}_{sh})+\tilde{g}_{pl}\tilde{g}^{sh}\delta_{q}(\Gamma(h)^{l}_{sh})+\tilde{g}_{ql}
\Gamma(\tilde{g})^{l}_{pw}\xi^{w}
+\tilde{g}_{pl}\Gamma(\tilde{g})^{l}_{qw}\xi^{w}].\nonumber
\end{align}
Applying  twice the vector field  $\frac{\delta}{\delta x^{s}}$ on the components of metric tensor  $g_{pq}$ yields
\begin{eqnarray*} \label{RE1}
\delta_{s}\delta_{h}g_{pq}&=&\frac{\partial^{2}g_{pq}}{\partial x^{s}\partial x^{h}}-\frac{\partial N^{k}_{h}}{\partial x^{s}}\frac{\partial g_{pq}}{\partial y^{k}}-N^{k}_{h}\frac{\partial^{2}g_{pq}}{\partial x^{s}\partial y^{k}}-N^{l}_{s}\frac{\partial^{2}g_{pq}}{\partial y^{l}\partial x^{h}}\nonumber\\
&+&N^{l}_{s}\frac{\partial N_{h}^{k}}{\partial y^{l}}\frac{g_{pq}}{\partial y^{k}}+
N^{l}_{s}N_{h}^k\frac{\partial^2g_{pq}}{\partial y^k\partial y^l}.\nonumber
\end{eqnarray*}
Finally we get
\begin{eqnarray} \label{1000}
&y^{p}y^{q}&[\frac{\partial}{\partial t}\tilde{g}_{pq}(t)-\tilde{g}^{sh}\frac{\delta^{2}\tilde{g}_{pq}}{\delta x^{s}\delta x^{h}}-\tilde{F}^2\tilde{g}^{sh}\frac{\partial^{2}\tilde{g}_{pq}}{\partial y^{s}\partial y^{h}}\\
&&+2(\Gamma^{s}_{hs}\Gamma^{h}_{pq}-\Gamma^{s}_{hq}\Gamma^{h}_{ps})+(\delta_{s}\tilde{g}^{sh})(\delta_{p}\tilde{g}_{hq}+\delta_{q}\tilde{g}_{ph}-\delta_{h}\tilde{g}_{pq})\nonumber\\
&&-(\delta_{q}\tilde{g}^{sh})(\delta_{p}\tilde{g}_{hs}+\delta_{s}\tilde{g}_{ph}-\delta_{h}\tilde{g}_{ps})\nonumber\\
&&+\tilde{g}_{ql}(\delta_{p}\tilde{g}^{sh})(\Gamma(h)^{l}_{sh}-\Gamma(\tilde{g})^{l}_{sh})+\tilde{g}_{pl}(\delta_{q}\tilde{g}^{sh})(\Gamma(h)^{l}_{sh}-\Gamma(\tilde{g})^{l}_{sh})\nonumber\\
&&-\tilde{g}_{ql}\tilde{g}^{sh}\delta_{p}(\Gamma(h)^{l}_{sh})-\tilde{g}_{pl}\tilde{g}^{sh}\delta_{q}(\Gamma(h)^{l}_{sh})-\tilde{g}_{ql}\Gamma
(\tilde{g})^{l}_{pw}\xi^{w}-\tilde{g}_{pl}
\Gamma(\tilde{g})^{l}_{qw}\xi^{w}]=0.\nonumber
\end{eqnarray}
The expression in the brackets in (\ref{1000}) is written in the following form
\begin{equation}\label{1001}
\frac{\partial}{\partial t}\tilde{g}_{pq}(t)-\tilde{g}^{sh}\frac{\delta^{2}\tilde{g}_{pq}}{\delta x^{s}\delta x^{h}}-\tilde{F}^2\tilde{g}^{sh}\frac{\partial^{2}\tilde{g}_{pq}}{\partial y^{s}\partial y^{h}}+\textrm{lower order terms}=0.
\end{equation}
This is a quasilinear system of strictly parabolic differential equations on $TM$ which is homogeneous of degree zero. 
By assumption $M$ is compact and the sphere bundle $SM$ as well. Therefore we rewrite (\ref{1001}) on $SM$. By restricting the metric tensor $\tilde{g}$ on $p^{*}TM$ and using Lemma \ref{mm} we can rewrite (\ref{1001}) in terms of the basis $\{\hat{e}_a,\hat{e}_{n+\alpha}\}$ on $SM$ as follows
\begin{eqnarray}\label{asli}
\frac{\partial}{\partial t}\tilde{g}_{pq}(t)-\tilde{g}^{ab}\hat{e}_{b}\hat{e}_{a}\tilde{g}_{pq}-
\tilde{g}^{\alpha\beta}\hat{e}_{n+\beta}\hat{e}_{n+\alpha}\tilde{g}_{pq}
+B^c\hat{e}_c\tilde{g}_{pq}
+D^\gamma\hat{e}_{n+\gamma}\tilde{g}_{pq}\nn\\
+\textrm{lower order terms}=0,
\end{eqnarray}
Since the 
coefficients
\begin{displaymath}
G^{_{AB}}=\left(\begin{array}{c|c}
\tilde{g}^{ab} & 0 \\
\hline
0 & \tilde{g}^{\alpha\beta}
\end{array}\right)_{(2n-1)\times(2n-1)},
\end{displaymath}
of principal (second) order terms of (\ref{asli}) are  positive definite, by definition, it is a quasilinear strictly parabolic system on $SM$. By means of the standard existence and uniqueness theorem for parabolic systems, (\ref{asli}) has a unique solution, namely $\tilde{g}_{pq}$ on $SM$. Recall that this solution is by definition a member of the family of Finsler metrics defined on $p^{*}TM$ and determines a Finsler structure $\tilde{F}$ on $TM$. Therefore the Finslerian Ricci-DeTurck flow (\ref{1000}) has a solution on $TM$.
This completes the existence statement.
\end{proof}
\section{Short time solution to the Finslerian Ricci flow}
In this section it is shown that any solution to the Finslerian Ricci-DeTurck flow gives rise to a solution to the Finslerian Ricci flow. Here, we recall some well known results which will be used in the sequel.
\begin{Lemma} \label{main9}
\cite{Chow1} If $\{X_{t}:0\leq t<T\leq \infty\}$ is a continuous time-dependent family of vector fields on a compact manifold $M$, then there exists a one-parameter family of diffeomorphisms $\{\varphi_{t}:M \longrightarrow M;\quad 0 \leq t<T \leq \infty\}$ defined on the same time interval such that
\begin{equation} \label{lem}
\left\{
\begin{array}{l}
\frac{\partial \varphi_{t}}{\partial t}(x)=X_{t}[\varphi_{t}(x)],\cr
\varphi_{0}(x)=x,
\end{array}
\right.
\end{equation}
for all $x\in M$ and $t\in[0,T)$.
\end{Lemma}
\begin{propriete} \label{main10}
\cite{Jost} Let $M$ be a smooth manifold and $V:J\times M\longrightarrow TM$ a smooth time-dependent vector field on $M$. Then there exists an open set $U\subset J\times I\times M$ and a smooth map $\phi:U\longrightarrow M$ such that for each $s\in J$ and $p\in M$, the set $U^{(s,p)}=\{t\in J:(t,s,p)\in U\}$ is an open interval containing $s$ and the smooth curve
\begin{equation} \label{43}
\left\{
\begin{array}{l}
\gamma :U^{(s,p)}\longrightarrow M,\cr
\gamma(t)=\phi(t,s,p),
\end{array}
\right.
\end{equation}
is the unique maximal solution to the initial value problem
\begin{equation}\label{44}
\left\{
\begin{array}{l}
\dot{\gamma}(t)=V(t,\gamma(t)),\cr
\gamma(s)=p.
\end{array}
\right.
\end{equation}
\end{propriete}
\begin{rem}\label{remark1}
If $M$ is compact, then the sphere bundle $SM$ is compact as well. On the other hand the components of $\xi=\Delta_{\tilde{g}(t),h}Id=\tilde{g}^{mn}(x,y)(-\tilde{\Gamma}^{k}_{mn}+\bar{\Gamma}^{k}_{mn})\frac{\partial}{\partial x^k}\in \Gamma(\pi^{*}TM)$ are homogeneous of degree zero on $y$ and  $\xi$ can be considered as a vector field on $SM$. According to Lemma \ref{main9} and Theorem \ref{main10}, there exists a unique one-parameter family of diffeomorphisms $\{\tilde{\varphi}_{t}:SM\longrightarrow SM \}$, $0\leq t<T$, such that
\begin{equation} \label{45}
\left\{
\begin{array}{l}
\frac{\partial}{\partial t}\tilde{\varphi}_{t}(z)=\xi(\tilde{\varphi}_{t}(z),t),\cr
\varphi_{0}=Id_{SM},
\end{array}
\right.
\end{equation}
where $z=(x,[y])\in SM$.
\end{rem}
\begin{Lemma} \label{LEM1}
\cite{BY}Let $M$ be a differentiable manifold, $F$ a Finsler structure and $\tilde{\varphi}_{t}$ a family of diffeomorphisms on $TM$ generated by a vector field $\hat{V}$. Then the pull back of $F$ under the point transformation i.e., $\tilde{\varphi}_{t}^{*}(F)$ is also a Finsler structure on $TM$. Moreover, if $\mathcal{R}ic_{F}$ is the Ricci scalar related to the Finsler structure $F$, then we have
\begin{eqnarray*}
\tilde{\varphi}_{t}^{*}(\mathcal{R}ic_{F})=\mathcal{R}ic_{\tilde{\varphi}_{t}^{*}(F)}.\nonumber
\end{eqnarray*}
\end{Lemma}
Now we are in a position to prove the following proposition.
\begin{prop} \label{main11}
Fix a compact Finsler manifold $(M,F)$ with related Finsler metric tensor $h$. Let $\tilde{F}(t)$, for $t\in[0,T)$, be a one-parameter family of Finsler structures on $TM$ which satisfies the following Finslerian Ricci-DeTurck flow
\begin{eqnarray} \label{46}
\frac{\partial}{\partial t}\tilde{F}^{2}(t)=-2\tilde{F}^{2}(t)\mathcal{R}ic(\tilde{F}(t))-\mathcal{L}_{\xi}\tilde{F}^{2}(t),
\end{eqnarray}
where, $\xi=\Delta_{\tilde{g}(t),h}Id$. Moreover let $\tilde{\varphi}_{t}$, for $t\in[0,T)$ be a one-parameter family of diffeomorphisms satisfying
\begin{eqnarray*} 
\frac{\partial}{\partial t}\tilde{\varphi}_{t}(x,y)=\xi\mid_{\varphi_{t}(x)}=(\Delta_{\tilde{g}(t),h}Id)\mid_{\varphi_{t}(x)},\nonumber
\end{eqnarray*}
for all points $x\in M$, $y\in S_{x}M$ and all $t\in[0,T)$. Then the Finsler structures $F(t)$, $t\in[0,T)$, form a solution to the Finslerian Ricci-flow (\ref{20}), where $F(t)$ is defined by
\begin{eqnarray*} 
F^{2}(t)=\tilde{\varphi}_{t}^{*}(\tilde{F}^{2}(t)).\nonumber
\end{eqnarray*}
\end{prop}
\begin{proof}
We are going to show $\frac{\partial}{\partial t}(\log F(t))=-\mathcal{R}ic.$
Derivation of $F^{2}(t)=\tilde{\varphi}_{t}^{*}(\tilde{F}^{2}(t))$ with respect to the parameter $t$, leads to
\begin{equation} \label{50}
\frac{\partial}{\partial t}(\log F(t))=\frac{1}{2}\frac{\frac{\partial}{\partial t}(\tilde{\varphi}_{t}^{*}(\tilde{F}^{2}(t)))}{\tilde{\varphi}_{t}^{*}(\tilde{F}^{2}(t))}.
\end{equation}
The term $\frac{\partial}{\partial t}(\tilde{\varphi}_{t}^{*}\tilde{F}^{2}(t))$ is
\begin{eqnarray} \label{51}
\frac{\partial}{\partial t}(\tilde{\varphi}_{t}^{*}\tilde{F}^{2}(t))&=&\frac{\partial}{\partial s}(\tilde{\varphi}_{s+t}^{*}(\tilde{F}^{2}(s+t)))\mid_{s=0}\nonumber \\
&=&\tilde{\varphi}_{t}^{*}(\frac{\partial}{\partial t}\tilde{F}^{2}(t))+\mathcal{L}_{\frac{\partial}{\partial s}(\tilde{\varphi}_{t}^{-1}\circ\hspace{0.1cm}\tilde{\varphi}_{t+s})\mid_{s=0}}\tilde{\varphi}_{t}^{*}(\tilde{F}^{2}(t)).
\end{eqnarray}
Hence, (\ref{51}) is written
\begin{eqnarray*}
\frac{\partial}{\partial t}(\tilde{\varphi}_{t}^{*}\tilde{F}^{2}(t))=\tilde{\varphi}_{t}^{*}(\frac{\partial}{\partial t}\tilde{F}^{2}(t))+\mathcal{L}_{(\tilde{\varphi}_{t}^{-1})_{*}(\xi)}\tilde{\varphi}_{t}^{*}(\tilde{F}^{2}(t)).
\end{eqnarray*}
Replacing the last relation in (\ref{50}) and by assumption (\ref{46}) we get
\begin{eqnarray*}
\frac{\partial}{\partial t}(\log F(t))&=&\frac{1}{2}\frac{\tilde{\varphi}_{t}^{*}(\frac{\partial}{\partial t}\tilde{F}^{2}(t))+\mathcal{L}_{(\tilde{\varphi}_{t}^{-1})_{*}(\xi)}
\tilde{\varphi}_{t}^{*}(\tilde{F}^{2}(t))}{\tilde{\varphi}_{t}^{*}(\tilde{F}^{2}(t))} \\
&=&\frac{1}{2}\frac{-2\tilde{\varphi}_{t}^{*}(\tilde{F}^{2}(t))\tilde{\varphi}_{t}^{*}(\mathcal{R}ic(\tilde{F}(t)))}{\tilde{\varphi}_{t}^{*}(\tilde{F}^{2}(t))}.
\end{eqnarray*}
By virtue of Lemma \ref{LEM1} we have
\begin{eqnarray*}
\frac{\partial}{\partial t}(\log F(t))
=-\tilde{\varphi}_{t}^{*}(\mathcal{R}ic(\tilde{F}(t)))
=-\mathcal{R}ic_{\tilde{\varphi}_{t}^{*}(\tilde{F}(t))}
=-\mathcal{R}ic_{F(t)}.
\end{eqnarray*}
Therefore, the Finsler structures $F(t)$ form a solution to the Finslerian Ricci flow. This completes the proof.
\end{proof}
Now we are in a position to prove our last theorem, influenced by the proof in Riemannian case, see \cite{BRE}
\begin{them} \label{main14}
Let $M$ be a compact differentiable manifold with a fixed background Finsler structure $\bar{F}$ on $TM_{0}$ and related Finsler metric $h$. Given any initial Finsler structure $F_{0}$ with metric tensor $g_{0}$, there exists a real number $T>0$ and a smooth one-parameter family of Finsler structures $F(t)$, $t\in[0,T)$, with metric tensors $g(t)$, such that $F(t)$ is a solution to the Finslerian Ricci flow and $F(0)=F_{0}$.
\end{them}
\begin{proof}
It is easy to check the existence statement. In fact by means of Theorem \ref{main8}, there exists a solution $\tilde{F}(t)$ to the Finslerian Ricci-DeTurck flow (\ref{22}) which is defined on some time interval $[0,T)$ and satisfies $\tilde{F}(0)=F_{0}$. For each point $z=(x,[y])\in SM$, we denote by $\tilde{\varphi}_{t}(z)$ the solution of the PDE
\begin{eqnarray*} 
\frac{\partial}{\partial t}\tilde{\varphi}_{t}(x,y)=(\Delta_{\tilde{g}(t),h}Id)\mid_{\varphi_{t}(x)}=\xi\mid_{\varphi_{t}(x)},\nonumber
\end{eqnarray*}
with the initial condition $\tilde{\varphi}_{0}(z)=z$. By Proposition \ref{main11}, the Finsler structures $F^{2}(t)=\tilde{\varphi}^{*}_{t}(\tilde{F}^{2}(t))$, $t\in[0,T)$, form a solution to the Finslerian Ricci flow (\ref{20}) with $F(0)=F_{0}$. This completes the proof.\\
\end{proof}

\bibliography{mybibfile}

\begin{thebibliography} {WWW99}
\bibitem{AZ} H. Akbar-Zadeh, Sur les espaces de Finsler a courbures sectionnelles constantes, Acad, Roy. Belg. Bull. Cl. sci. (5) 74 (1988), 281-322.
\bibitem{Bao} D. Bao, On two curvature-driven problems in Riemann-Finsler geometry: In memory of Makoto Matsumoto, Advanced studies in pure mathematics, Vol. 48, Mathematical Society, Japan, Tokyo,  (2007), pp. 19-71.
\bibitem{BCS} D. Bao, S. Chern, Z. Shen, An introduction to Riemann-Finsler Geometry. Graduate Texts in Mathematics, Vol. 200, Springer, 2000.
\bibitem{BAOS} D. Bao and Z. Shen, On the volume of unit tangent spheres in a Finsler manifold, Results in Mathematics, Vol. 26, (1994), 1-17.
\bibitem{BS} B. Bidabad, Z. Shen, Circle-preserving transformations in Finsler spaces. Publ. Math. Debrecen, 81, 3-4 (2012), 435-445.
\bibitem{BY} B. Bidabad and M. Yar Ahmadi, {On quasi-Einstein Finler spaces}, Bulletin of the Iranian Mathematical Society, vol. \textbf{40}, no. 4, (2014),  921-930.
\bibitem{YB2} B.  Bidabad and M. Yar Ahmadi, {Convergence of Finslerian metrics under Ricci flow}, to appear in Sci. China Math., (2015).
\bibitem{YB1} B. Bidabad and M. Yar Ahmadi, {On compact Ricci solitons in Finsler geometry},	arXiv:1508.02148 (2015).
\bibitem{BRE} S. Brendle, Ricci Flow and the Sphere Theorem, Vol. III, Graduate Studies in Mathematics, AMS, 2010.
\bibitem{Cao} H. D. Cao, X. P. Zhu, Hamilton-Perelman's proof of the Poincar\'{e} conjecture and the geometrization conjecture, Asian J. Math, 10, 2, (2006),  165-492.
\bibitem{Chow1} B. Chow, D. Knopf, The Ricci Flow: An Introduction, Mathematical Survays and Monographs, Vol. 110, AMS,  Providence, RI, 2004.
\bibitem{Chow2} B. Chow, P. Lu, L. Ni, Hamilton's Ricci Flow, Graduate studies in Mathematics, Vol. 77 AMS, Providence, IR, 2006.
\bibitem{DeT} D. M. DeTurck, Deforming metrics in the direction of their Ricci tensors, J. Differential Geom. 18(1) (1983), 157--162.
\bibitem{Ham1} R. S. Hamilton, Three manifolds with positive Ricci curvature, J. Differential Geom. 17, (1982), 255-306.
\bibitem{JB} P. Joharinad, B. Bidabad, Conformal vector fields on Finsler spaces, Differential Geometry and its Applications, 31, (2013), 33--40.
\bibitem{SL} S. Lakzian, {Differential Harnack estimates for heat equation under Finsler-Ricci flow}, to appear in Pacific Journal of Mathematics, (2015).
\bibitem{MG} J. M. Morgan, G. Tian, Ricci flow and the Poincar\'{e} conjecture, arxiv: math DG$\backslash$0607607.
\bibitem{Pr1} G. Perelman, Ricci flow with surgery on three-manifolds, arxive:math.DG$\backslash$03109.
\bibitem{Pr2} G. Perelman, The entropy formula for the Ricci flow and its geometric applications,
     arxive.org$\backslash$math.DG$\backslash$02111159.
\bibitem{Rog} M. Renardy, R. C. Rogers, An introduction to partial differential equations, Texts in applied mathematics, second edition, Springer 2004.
\end{thebibliography}
{\small \it  Behroz Bidabad, bidabad@aut.ac.ir}\\
{\small \it Maral Khadem Sedaghat,
m\_sedaghat@aut.ac.ir \\ Faculty of Mathematics and Computer Science, Amirkabir University of Technology (Tehran Polytechnic), Hafez Ave., 15914 Tehran, Iran.}

\end{document}